\documentclass[12pt]{amsart}
\usepackage{amsmath}
\usepackage{amssymb}
\usepackage[all]{xy}
\usepackage{longtable}
\usepackage[osf,sc]{mathpazo}
\usepackage{euscript}
\usepackage{calrsfs}
\usepackage{color}
\usepackage{multirow,bigdelim}

\newtheorem{theorem}[equation]{Theorem}
\newtheorem{lemma}[equation]{Lemma}
\newtheorem{proposition}[equation]{Proposition}
\newtheorem{corollary}[equation]{Corollary}

\newtheorem{definition}[equation]{Definition}

\newtheorem{example}[equation]{Example}

\theoremstyle{remark}
\newtheorem{remark}[equation]{Remark}

\numberwithin{equation}{subsection}

\allowdisplaybreaks[1]

\newcommand{\FF}{\mathbb{F}}
\newcommand{\ZZ}{\mathbb{Z}}
\newcommand{\QQ}{\mathbb{Q}}

\newcommand{\GG}{\mathbb{G}}

\newcommand{\CC}{\mathbb{C}}

\newcommand{\NN}{\mathbb{N}}

\newcommand{\bg}{\mathbf{g}}

\newcommand{\bx}{\mathbf{x}}

\newcommand{\bu}{\mathbf{u}}
\newcommand{\bv}{\mathbf{v}}

\newcommand{\bC}{\mathbf{C}}

\newcommand{\cF}{\mathcal{F}}

\DeclareMathAlphabet{\matheur}{U}{eur}{m}{n}

\newcommand{\fs}{\mathfrak{s}}
\newcommand{\fm}{\mathfrak{m}}

\newcommand{\rtr}{\mathrm{tr}}

 \DeclareMathOperator{\Lie}{Lie}
\DeclareMathOperator{\Ker}{Ker} 
\DeclareMathOperator{\Mat}{Mat}

 \DeclareMathOperator{\wt}{wt}
\DeclareMathOperator{\Ext}{Ext}  
\DeclareMathOperator{\Li}{Li}

\DeclareMathOperator{\ad}{ad}

\newcommand{\ok}{\overline{k}}

\newcommand{\tr}{\mathrm{tr}}

\newcommand{\laurent}[2]{{#1 (\!( #2 )\!)}}

\begin{document}

\title[$v$-adic vanishing versus $\infty$-adic Eulerianness ]{{\large{O\MakeLowercase{n multiple polylogarithms in characteristic $p$: $v$-adic vanishing versus $\infty$-adic} E\MakeLowercase{ulerianness}}}}

\author{Chieh-Yu Chang}
\address{Department of Mathematics, National Tsing Hua University, Hsinchu City 30042, Taiwan
  R.O.C.}

\email{cychang@math.nthu.edu.tw}

\author{Yoshinori Mishiba }
\address{Department of Life, Environment and Materials Science,
Fukuoka Institute of Technology, Japan}
\email{mishiba@fit.ac.jp }

\thanks{The first author was partially supported by a Golden-Jade
  fellowship of the Kenda Foundation and MOST Grant
  102-2115-M-007-013-MY5.  }
\thanks{The second author was partially supported by JSPS KAKENHI Grant Number 15K17525.}


\date{April 5, 2017}

\begin{abstract} In this paper, we give a simultaneous vanishing principle for the $v$-adic Carlitz multiple polylogarithms (abbreviated as CMPLs) at algebraic points, where $v$ is a finite place of the rational function field over a finite field. This principle establishes the fact that the $v$-adic vanishing of CMPLs at algebraic points is equivalent to its $\infty$-adic counterpart being Eulerian. This reveals a nontrivial connection between the $v$-adic and $\infty$-adic worlds in positive characteristic.
\end{abstract}


\maketitle
\section{Introduction}
\subsection{Motivation}

The study of this paper is motivated by the classical theory and conjectures for special zeta values. Let $n\geq 2$ be an integer. Then the celebrated formula of Euler for special values of the Riemann zeta function at even positive integers implies that
\[ \zeta(n)/(2\pi \sqrt{-1})^{n}\in \QQ \Leftrightarrow n\hbox{ is even}.\]
 For a prime number $p$, we consider the Kubota-Leopoldt $p$-adic zeta function $\zeta_{p}$, that interpolates the rational values of the Riemann zeta function at non-positive integers. When $n$ is even, we know that $\zeta_{p}(n)=0$ (cf.~\cite{KL64, Col82, F04}). So this reveals the following interesting phenomena between the archimedean world and its $p$-adic counter part: for an integer $n\geq 2$,
 \[ \zeta(n)/(2\pi \sqrt{-1})^{n}\in \QQ\Leftrightarrow n\hbox{ is even } \Rightarrow \zeta_{p}(n)=0    .\]
 Conjecturally, the reverse direction above is valid. Let $\NN$ be the set of positive integers.
It is known by \cite{S81} that for $m\in \NN$, $\zeta_{p}(2m+1)$ is nonzero when $p$ is regular or $(p-1)|2m$.

We note that for an integer $n\geq2$, $\zeta(n)=\Li_{n}(1)$, where $\Li_{n}$ is the $n$th polylogarithm given by
\[\Li_{n}(z):=\sum_{m=1}^{\infty}\frac{z^{m} }{m^{n} } .\] For a fixed prime number $p$, we let $\CC_{p}$ be the $p$-adic completion of a fixed algebraic closure of the $p$-adic numbers $\QQ_{p}$. The series $\Li_{n}(z)$ converges $p$-adically on the open unit disc and we denote this function by $\Li_{n, p}$ which is called the $p$-adic $n$th polylogarithm. For each branch parameter $a\in \CC_{p}$ of the $p$-adic logarithm, we note that by \cite{F04, Col82} $\Li_{n,p}$ can be analytically continued to $\CC_{p} \smallsetminus \{ 1 \}$, and we denote by $\Li_{n,p}^{a}$ its analytically continued function. Furusho~\cite{F04} showed that $\zeta_{p}(n)$ is the limit of $\Li_{n,p}^{a}$ when ``$z\rightarrow 1$'' in the sense of \cite{F04}, and $\zeta_{p}(n)$ is independent of choices of branch parameters $a$.

We fix embeddings $\overline{\QQ} \hookrightarrow \CC$ and $\overline{\QQ} \hookrightarrow \CC_{p}$. Let $u$ be a nonzero algebraic number for which $\Li_{n}(u)$ converges with respect to the archimedean absolute value.  Inspired by the conjectural equivalence between the rationality of $\zeta(n)/(2\pi \sqrt{-1})^{n}$ and the vanishing of  $\zeta_{p}(n)$, for a fixed branch parameter $a\in\CC_{p}$ of the $p$-adic logarithm one naturally asks  whether there is a criterion $\spadesuit_{a}$ for the rationality of $\Li_{n}(u)/(2\pi \sqrt{-1})^{n}$ in terms of the vanishing of $\Li_{n,p}^{a}(u)$:
\[ \Li_{n}(u)/(2\pi\sqrt{-1})^{n}\in \QQ \Leftrightarrow   \spadesuit_{a}   \Leftrightarrow   \Li_{n,p}^{a}(u)=0 .\]
For example, for $n \geq 2$ and ``$u \to 1$'', conjecturally $\spadesuit_{a}=$ \lq\lq$n \hbox{ is even}$\rq\rq.

 The main purpose of this paper is to give a positive answer of the analogous question above for the Carlitz multiple polylogarithms (abbreviated as CMPLs), which was introduced by the first author of the present paper in \cite{C14} generalizing the notion initiated by Anderson and Thakur~\cite{AT90}.

\subsection{Depth one case in function fields}

Let $A:=\FF_{q}[\theta]$ be the polynomial ring in the variable $\theta$ over the finite field $\FF_{q}$ of $q$ elements, and $k$ be its quotient field. Denote by $\infty$ the infinite place of $k$ with an associated absolute value $|\cdot|_{\infty}$. Let $k_{\infty}:=\laurent{\FF_{q}}{1/\theta}$ be the $\infty$-adic completion of $k$. We fix an algebraic closure $\overline{k_{\infty}}$ of $k_{\infty}$, and let $\CC_{\infty}$ be the $\infty$-adic completion of $\overline{k_{\infty}}$. We further fix a fundamental period $\tilde{\pi}\in \overline{k_{\infty}}^{\times}$ of the Carlitz $\FF_{q}[t]$-module $\bC$, where $t$ is an independent variable. Note that $\bC$ and $\tilde{\pi}$ play the analogous roles of $\mathbb{G}_{m}$ and $2\pi \sqrt{-1}$ respectively in the function field setting. A positive integer $n$ is called $A$-even if $(q-1)|n$, as $q-1$ is the cardinality of the unit group $A^{\times}$.

\subsubsection{Special zeta values}
Let $A_{+}$ be the set of monic polynomials in $A$ and consider the Carlitz zeta value at $n\in \NN$,
\[ \zeta_{A}(n):=\sum_{a\in A_{+}}\frac{1}{a^{n}}\in k_{\infty}^{\times}  .\]
In \cite{Ca35}, Carlitz derived an analogue of Euler's formula on values of the Riemann zeta function at even positive integers. More precisely, we have the following consequence: for a positive integer $n$, we have the equivalence
\[   \zeta_{A}(n)/\tilde{\pi}^{n}\in k\Leftrightarrow n \hbox{ is } A\hbox{-even}.   \]

Let $v$ be a monic prime of $A$, and let $k_{v}$ be the completion of $k$ with respect to the normalized $v$-adic absolute value $| \cdot |_{v}$ associated to the place $v$. We fix an algebraic closure $\overline{k_{v}}$ of $k_{v}$, and let $\CC_{v}$ be the $v$-adic completion of $\overline{k_{v}}$. Let $\bar{k}$ be an algebraic closure of $k$, and fix the natural embeddings $\bar{k}\hookrightarrow \overline{k_{\infty}}$ and $\bar{k}\hookrightarrow \overline{k_{v}}$ as $k\subseteq k_{\infty}$ and $k\subseteq k_{v}$. For a positive integer $n$, in analogy with the special value at $n$ of the Kubota-Leopoldt $p$-adic zeta function $\zeta_{p}(n)$ we consider the $v$-adic Goss zeta function at $n$ denoted by $\zeta_{A,v}(n) \in k_{v}$ (see \cite{Go79}). Goss~\cite{Go79} showed that $\zeta_{A,v}(n)$ vanishes for $A$-even $n$, and Yu~\cite{Yu91} showed the transcendence of $\zeta_{A,v}(n)$ for $A$-odd $n$ (ie., $(q-1)\nmid n$), and therefore we have the following complete story: for $n\in \NN$,
\begin{equation}\label{E:CorrZeta}
 \zeta_{A}(n)/\tilde{\pi}^{n}\in k\Leftrightarrow n\hbox{ is }A\hbox{-even} \Leftrightarrow \zeta_{A,v}(n)=0.
\end{equation}
However, Yu~\cite{Yu91} showed that $\zeta_{A}(n)/\tilde{\pi}^{n}\in\bar{k}\Leftrightarrow \zeta_{A}(n)/\tilde{\pi}^{n}\in k$ and hence we have the following equivalence:
\[ \zeta_{A}(n)/\tilde{\pi}^{n}\in \bar{k}\Leftrightarrow n\hbox{ is }A\hbox{-even} \Leftrightarrow \zeta_{A,v}(n)=0.\]

\subsubsection{Carlitz polylogarithms} Put $L_{0}:=1$ and $L_{i}:=(\theta-\theta^{q})\cdots(\theta-\theta^{q^{i}})$ for $i\in \NN$. Let $\log_{\bC}$ be the logarithm of the Carlitz $\FF_{q}[t]$-module $\bC$ given by the power series
\[ \log_{\bC}(z):=\sum_{i=0}^{\infty}\frac{z^{q^{i}}}{L_{i}} \hbox{ (see \cite{Go96, T04})} .\]
In analogy with the classical polylogarithm, Anderson and Thakur~\cite{AT90} defined the $n$th Carlitz polylogarithm for $n\in \NN$:
\[ \Li_{n}(z):=\sum_{i=0}^{\infty} \frac{z^{q^{i}}}{L_{i}^{n}}   .\]
Unlike the simple identity between $\zeta(n)$ and the $n$th polylogarithm at $1$ in the classical case, Anderson and Thakur~\cite{AT90} proved that $\zeta_{A}(n)$ is a $k$-linear combination of $\Li_{n}$ at some explicit integral points in $A$.

 For a positive integer $n$, we denote by $\bC^{\otimes n}$ the $n$th tensor power of the Carlitz module $\bC$ introduced by Anderson-Thakur~\cite{AT90}. Let $u\in \ok^{\times}$ and put $\bv:=(0,\ldots,0,u)^{\tr}\in \bC^{\otimes n}(\ok)$. We further put the condition:
\[ \clubsuit_{1}:=`` \bv=(0,\ldots,0,u)^{\tr} {\hbox{ is an }}\FF_{q}[t]{\hbox{-torsion in }}\bC^{\otimes n}(\bar{k})"   .\]
We assume $|u|_{\infty}<|\theta|_{\infty}^{\frac{nq}{q-1}}$ and note that $\Li_{n}(u)$ converges in $\CC_{\infty}$.
Then combining the theories of Anderson-Thankur ~\cite{AT90} and Yu ~\cite{Yu91} derives the following equivalence:
\begin{equation}\label{E:PlogCorr1}
\Li_{n}(u)/\tilde{\pi}^{n}\in \bar{k} \Leftrightarrow \clubsuit_{1}\hbox{ is valid}.
\end{equation}

Unlike the analytic continuation of classical $p$-adic polylogarithms, we shall mention here that the convergence domain of $\Li_{n}$ with respect to the $v$-adic absolute value is the open unit disc $\left\{z\in \CC_{v};|z|_{v}<1 \right\}$ and we only extend its convergence domain to the $v$-adic closed unit disc (see Definition~\ref{analytic-continuation-CMPLs} and Remark~\ref{Rmk:analytic-continuation-CMPLs}). We denote by $\Li_{n}(z)_{v}$ the extended function of $\Li_{n}(z)$ on the $v$-adic closed unit disc. Let $u\in\ok^{\times}$ satisfy $|u|_{v}\leq 1$ (so $\Li_{n}(u)_{v}$ is defined). Then by Yu's theory~\cite{Yu91} (see Theorem~\ref{T:VanishingCriterion} also) one can derive the following equivalence
\begin{equation}\label{E:PlogCorr2}
  \Li_{n}(u)_{v}=0 \Leftrightarrow \clubsuit_{1}\hbox{ is valid},
\end{equation}
and hence one establishes the principle: for $u\in \bar{k}^{\times}$ with $|u|_{\infty}<|\theta|_{\infty}^{\frac{nq}{q-1}}$ and $|u|_{v}\leq 1$ we have
\[  \Li_{n}(u)/\tilde{\pi}^{n}\in \bar{k}\Leftrightarrow \clubsuit_{1}\hbox{ is valid}\Leftrightarrow    \Li_{n}(u)_{v}=0 .\] The main result in this paper is to generalize this principle to higher depths described below.

\subsection{Higher depths case} In analogy with the classical multiple polylogarithms, the first author of the present paper introduced the Carlitz multiple polylogarithm (abbreviated as CMPL) for each $\fs=(s_{1},\ldots,s_{r})\in \NN^{r}$ (see~\cite{C14}):
\[\Li_{\fs}(z_{1},\ldots,z_{r}):=\sum_{i_{1}>\cdots>i_{r}\geq 0} \frac{z_{1}^{q^{i_{1}}}\cdots z_{r}^{q^{i_{r}}}}  {L_{i_{1}}^{s_{1}}\cdots L_{i_{r}}^{s_{r}}},  \] whose weight is defined to be ${\rm{}wt}(\fs):=\sum_{i=1}^{r}s_{i}$ and whose depth is defined to be $r$. It is shown in \cite{C14} that each multizeta value $\zeta_{A}(\fs)$ initiated by Thakur~\cite{T04} is a $k$-linear combination of $\Li_{\fs}$ at some integral points in $A^{r}$, generalizing the formula of Anderson-Thakur to the higher depth case.

 Fixing any $\fs=(s_{1},\ldots,s_{r})\in \NN^{r}$ and $\bu=(u_{1},\ldots,u_{r})\in (\ok^{\times})^{r}$, one has an associated $t$-module $G:=G_{\fs,\bu}$ defined over $\ok$ and an associated algebraic point $\bv:=\bv_{\fs,\bu}\in G(\ok)$ given in \cite{CPY14} (for avoiding the heavy notation, we drop $\fs$ and $\bu$ without confusion when it is clear from the context). Suppose that $|u_{i}|_{\infty} < |\theta|_{\infty}^{\frac{s_{i} q}{q-1}}$ for each $i=1,\ldots,r$. It is shown in \cite{CPY14} that
\[ \clubsuit_{r}:=``\bv\hbox{ is an }\FF_{q}[t]\hbox{-torsion point in }G(\ok)"\]
is valid if and only if
\[ \Li_{(s_{1},\ldots,s_{r})}(u_{1},\ldots,u_{r})/\tilde{\pi}^{s_{1}+\cdots+s_{r}},\Li_{(s_{2},\ldots,s_{r})}(u_{2},\ldots,u_{r})/\tilde{\pi}^{s_{2}+\cdots+s_{r}},\ldots,\Li_{s_{r}}(u_{r})/\tilde{\pi}^{s_{r}}       \]
are simultaneously in $k$. Note that the value $\Li_{\fs}(\bu)$ is called {\it{Eulerian}} if the ratio $\Li_{\fs}(\bu)/\tilde{\pi}^{\wt(\fs)}$ lies in $k$ (cf.~\cite{T04, CPY14}), and by \cite{C14} we know that $\Li_{\fs}(\bu)$ is Eulerian if and only if $\Li_{\fs}(\bu)/\tilde{\pi}^{\wt(\fs)}\in\ok$.

 Let $\mathcal{O}_{\CC_{v}}$ be the closed unit disc in $\CC_{v}$. In \S\S~\ref{Sec:Analytic continuation} we show that $\Li_{\fs}$ can be analytically continued to $\mathcal{O}_{\CC_{v}}^{r}$ and we denote by $\Li_{\fs}({\bf{z}})_{v}$ the extended function of $\Li_{\fs}$ for ${\bf{z}}\in \mathcal{O}_{\CC_{v}}^{r}$. Our main result, stated as Theorem~\ref{T:VanishingCriterion}, is to prove the $v$-adic counterpart fitting into the correspondence mentioned above. Given $(u_{1},\ldots,u_{r})\in (\ok^{\times})^{r}\cap \mathcal{O}_{\CC_{v}}^{r}$ with $|u_{i}|_{\infty} < |\theta|_{\infty}^{\frac{s_{i} q}{q-1}}$ for each $i$, we show that
$ \clubsuit_{r}$ is valid if and only if
\[ \Li_{(s_{1},\ldots,s_{r})}(u_{1},\ldots,u_{r})_{v}=\Li_{(s_{2},\ldots,s_{r})}(u_{2},\ldots,u_{r})_{v}=\cdots=\Li_{s_{r}}(u_{r})_{v}=0.\]  Note that when we restrict $r=1$, then we recover the result in (\ref{E:PlogCorr2}).

We shall mention that for the $\fs$ and $\bu$ given above, if $\Li_{(s_{1},\ldots,s_{r})}(u_{1},\ldots,u_{r})$ is Eulerian, then the following values \[\Li_{(s_{2},\ldots,s_{r})}(u_{2},\ldots,u_{r}),\ldots,\Li_{s_{r}}(u_{r})\]
are automatically Eulerian. This fact is proven in \cite{CPY14} using the ABP-criterion~\cite{ABP04}, which is a strong tool in the transcendence theory of $\infty$-adic case. However, we do not know whether its analog is true in the $v$-adic case, and additional work is necessary to develop $v$-adic transcendence theory.

A similar criterion for Thakur's multizeta values (abbreviated MZVs) to be {\it{Eulerian}} is given in \cite{CPY14}. One then naturally asks whether one has the criterion for its $v$-adic counterpart as $v$-adic MZVs were introduced in \cite{T04}. Such a  criterion is related to expressing the given $v$-adic MZV as the coordinate logarithm of a certain $t$-module, and it is not clear to the authors at this moment. In the depth one case, the far reaching theory of \cite{AT90} does relate both $\zeta_{A}(n)$ and $\zeta_{A,v}(n)$ to the logarithm of $\bC^{\otimes n}$, but for higher depth MZVs it is an open question although certain $\infty$-adic MZVs of higher depths are worked out in \cite{C15} by $t$-motivic methods.

The first step of proving the main result above is to write down the $t$-module $G$ and  the algebraic point $\bv$ explicitly. We mention that in \cite{CPY14}, $G$ and $\bv$ are only theoretically constructed without being written down explicitly. We further use some techniques in \cite{AT90} to compute the coefficient matrices of the logarithm of $G$ explicitly, and then show that the Carlitz multiple star polylogarithms (abbreviated as CMSPLs) defined in (\ref{E:star-CMPL}) occur as the coordinate logarithms of $G$. Then we establish an identity of the CMSPLs (stated as Lemma~4.1.3) in terms of linear combination of products of CMPLs and CMSPLs. These properties together with Yu's theory~\cite{Yu91} enable us to derive the desired results.

\subsection*{Acknowledgements}
We are grateful to H.~Furusho and J.~Yu for their discussions, which inspire this project. We thank M.~Papanikolas and D.~Thakur for their useful comments, and thank the referees for their suggestions, which greatly improve the exposition of this paper. The project was initiated when the second author visited NCTS and he would like to thank NCTS for their kind support.

\section{Preliminaries}
\subsection{Notation}

We adopt the following notation. \\
\begin{longtable}{p{0.5truein}@{\hspace{5pt}$=$\hspace{5pt}}p{5truein}}
$\FF_q$ & the finite field with $q$ elements, for $q$ a power of a
prime number $p$. \\
$\theta$, $t$ & independent variables. \\
$A$ & $\FF_q[\theta]$, the polynomial ring in the variable $\theta$ over $\FF_q$.
\\
$v$ & a monic irreducible polynomial in A.
\\
$k$ & $\FF_q(\theta)$, the fraction field of $A$.\\
$k_v$ &  the completion of $k$ with
respect to the place $v$.\\
$\overline{k_v}$ & a fixed algebraic closure of $k_v$.\\
$\ok$ & the algebraic closure of $k$ in $\overline{k_v}$.\\
$\CC_v$ & the completion of $\overline{k_v}$ with respect to
the canonical extension of $v$.\\
$|\cdot|_{v}$& a fixed absolute value on $\CC_{v}$ so that $|v|_{v}=1/q^{\tiny{\deg v}}$.\\
$\mathbb{G}_{a}$& the additive group scheme over $A$.\\

\end{longtable}

\subsection{CMPLs and CMSPLs} \label{subsection-CMPLs-CMSPLs}
We recall the Carlitz multiple polylogarithms \cite{C14} that are generalization of the polylogarithms initiated in \cite{AT90}. Put $L_{0}:=1$ and $L_{i}:=(\theta-\theta^{q})\cdots(\theta-\theta^{q^{i}})$ for $i\in \NN$. Given any $\fs:=(s_{1},\ldots,s_{r})\in \NN^{r}$, the associated Carlitz multiple polylogarithm (abbreviated as CMPL)
and the Carlitz multiple star polylogarithm (abbreviated as CMSPL, compared with the terminology in \cite{FKMT15}) are defined by the series

\[
 \Li_{\fs}(z_{1},\ldots,z_{r}):=\sum_{i_{1}>\cdots> i_{r}\geq 0} \frac{z_{1}^{q^{i_{1}}}\cdots z_{r}^{q^{i_{r}}}} {L_{i_{1}}^{s_{1}}\cdots L_{i_{r}}^{s_{r}}}
  \]
and
\begin{equation}\label{E:star-CMPL}
 \Li^{\star}_{\fs}(z_{1},\ldots,z_{r}):=\sum_{i_{1} \geq \cdots \geq i_{r}\geq 0} \frac{z_{1}^{q^{i_{1}}}\cdots z_{r}^{q^{i_{r}}}} {L_{i_{1}}^{s_{1}}\cdots L_{i_{r}}^{s_{r}}}.
  \end{equation}
We denote by $\Li_{\fs}(z_{1},\ldots,z_{r})_{v}$ and $\Li^{\star}_{\fs}(z_{1},\ldots,z_{r})_{v}$
while working with $v$-adic convergence. Both series converge when $|z_{1}|_{v} < 1$ and
$|z_{2}|_{v}, \dots, |z_{r}|_{v} \leq 1$. Indeed, since $v^{2}$ does not divide $\theta - \theta^{q^{i}}$, if $z_{i}$'s satisfy the above condition, then we have
\[
\left| \frac{z_{1}^{q^{i_{1}}}\cdots z_{r}^{q^{i_{r}}}} {L_{i_{1}}^{s_{1}}\cdots L_{i_{r}}^{s_{r}}} \right|_{v}
\leq |v|_{v}^{-(s_{1} i_{1} + \cdots + s_{r} i_{r})} |z_{1}|_{v}^{q^{i_{1}}} \cdots |z_{r}|_{v}^{q^{i_{r}}}
\leq |v|_{v}^{- \wt(\fs) i_{1}} |z_{1}|_{v}^{q^{i_{1}}}
\to 0 \ (i_{1} \to \infty).
\]

\subsection{$t$-modules} In this section, we review the theory of $t$-modules introduced by Anderson~\cite{A86}. For an $A$-algebra $R$, we denote by $\tau$ the Frobenius $q$th power operator $\tau:=\left( x\mapsto x^{q} \right):R\rightarrow R$. For convenience, we extend the action of $\tau$ on the matrices with entries in $R$ by componentwise action. We denote by $\CC_{v}[ \tau]$ the non-commutative polynomial ring generated by $\tau$ over $\CC_{v}$ subject to the relation
\[ \tau \alpha=\alpha^{q} \tau\hbox{ for }\alpha\in \CC_{v}.  \]For each $\varphi \in \Mat_{d}(\CC_{v}[\tau])$, we write $\varphi=\sum_{i=0}^{\infty}\alpha_{i}\tau^{i}$ with each $\alpha_{i}\in \Mat_{d}(\CC_{v})$ and $\alpha_{i}=0$ for $i\gg 0$, and further define $\partial \varphi:=\alpha_{0}$.

Let $t$ be a new variable and $d$ be a positive integer. By a $d$-dimensional $t$-module, we mean a pair $G=({\mathbb{G}_{a}^{d}},\rho)$, where $\rho$ is an $\FF_{q}$-linear ring homomorphism
\[ \rho:\FF_{q}[t]\rightarrow \Mat_{d}(\CC_{v}[\tau]) \]
so that $\partial \rho_{t}-\theta I_{d}$ is a nilpotent matrix. Here we denote by $\rho_{a}$ the image of $a \in \FF_{q}[t]$  by $\rho$. For a subring $A\subseteq R \subseteq \CC_{v}$, we say that the $t$-module is defined over $R$ if all the coefficient matrices of $\rho_{t}$ are in $\Mat_{d}(R)$. In this situation, $\GG_{a}^{d}(R)=R^{d}$ has an $\FF_{q}[t]$-module structure via the map $\rho$.

Let $K$ be either $\bar{k}$ or $\CC_{v}$ and let $G=(\GG_{a}^{d},\rho)$ be a $d$-dimensional $t$-module defined over $K$. Then we have the unique {\it{exponential function}} of $G$ which is an $\FF_{q}$-linear $d$-variable power series of the form $\exp_{G}=I_{d}+ \sum_{i=1}^{\infty}\alpha_{i} \tau^{i}$ with $\alpha_{i}\in \Mat_{d}(K)$, satisfying the following identity:
\begin{equation}\label{E:FunEquaExp}
 \exp_{G} \circ \partial \rho_{a}=\rho_{a}\circ \exp_{G} \hbox{ for all }a\in \FF_{q}[t].
\end{equation}
The logarithm of $G$ denoted by $\log_{G}$, is defined to be the formal inverse of $\exp_{G}$ that has the property:
\begin{equation}\label{E:FunEquaLog}
 \log_{G}\circ \rho_{a}=\partial \rho_{a} \circ \log_{G} \hbox{ for all }a\in \FF_{q}[t] .
\end{equation}
The logarithm $\log_{G}$ will be the primary interest for our study as it could provide a rich source of transcendental values (see~\cite{Yu91, Yu97}).

\section{Computation on the logarithms}
In this section, our goal is to give an explicit construction of an appropriate $t$-module $G$ over $\bar{k}$ associated to an index $\fs\in \NN^{r}$ and an algebraic point $\bu\in ({\ok}^{\times})^{r}$, and show that the CMSPLs in question occur as coordinate logarithms of $G$ at an explicit algebraic point of $G$.

\subsection{Construction of the $t$-module $G$}
In what follows, we fix $\fs = (s_{1}, \dots, s_{r})\in \NN^{r}$ and $\bu = (u_{1}, \dots, u_{r}) \in (\ok^{\times})^{r}$.
For $1 \leq \ell \leq r$,
we set $d_{\ell} := s_{\ell} + \cdots + s_{r}$ and $d := d_{1} + \cdots + d_{r}$.
Let $B$ be a $d \times d$-matrix of the form

\[
\left( \begin{array}{c|c|c}
B[11] & \cdots & B[1r] \\ \hline
\vdots & & \vdots \\ \hline
B[r1] & \cdots & B[rr]
\end{array} \right)
\]
where $B[\ell m]$ is a $d_{\ell} \times d_{m}$-matrix for each $\ell$ and $m$.
In this paper, $B[\ell m]$ is called the $(\ell, m)$-th block matrix of $B$.

For $1 \leq \ell \leq m \leq r$, we set

\[
N_{\ell} := \left(
\begin{array}{ccccc}
0 & 1 & 0 & \cdots & 0 \\
& 0 & 1 & \ddots & \vdots \\
& & \ddots & \ddots & 0 \\
& & & \ddots & 1 \\
& & & & 0
\end{array}
\right)
\in \Mat_{d_{\ell}}(\ok),
\]

\[
N := \left(
\begin{array}{cccc}
N_{1} & & & \\
& N_{2} & & \\
& & \ddots & \\
& & & N_{r}
\end{array}
\right)
\in \Mat_{d}(\ok),
\]

\[
E[\ell m] := \left(
\begin{array}{cccc}
0 & \cdots & \cdots & 0 \\
\vdots & \ddots & & \vdots \\
0 & & \ddots & \vdots \\
1 & 0 & \cdots & 0
\end{array}
\right)
\in \Mat_{d_{\ell} \times d_{m}}(\ok) \ \ \ (\mathrm{if} \ \ell = m),
\]

\[
E[\ell m] := \left(
\begin{array}{cccc}
0 & \cdots & \cdots & 0 \\
\vdots & \ddots & & \vdots \\
0 & & \ddots & \vdots \\
(-1)^{m-\ell} \prod_{e=\ell}^{m-1} u_{e} & 0 & \cdots & 0
\end{array}
\right)
\in \Mat_{d_{\ell} \times d_{m}}(\ok) \ \ \ (\mathrm{if} \ \ell < m),
\]

\[
E := \left(
\begin{array}{cccc}
E[11] & E[12] & \cdots & E[1r] \\
& E[22] & \ddots & \vdots \\
& & \ddots & E[r-1,r] \\
& & & E[rr]
\end{array}
\right)
\in \Mat_{d}(\ok).
\]
We also define

\[
E_{m} := \left( \begin{array}{c|c|c}
0 & 0 & 0 \\ \hline
0 & E[mm] & 0 \\ \hline
0 & 0 & 0
\end{array} \right) \in \Mat_{d}(\ok)
\]
to be the $d \times d$-matrix such that the $(m,m)$-th block matrix is $E[mm]$
and the others are zero matrices.

We define the $t$-module $G = G_{\fs, \bu} := (\GG_{a}^{d}, \rho)$ by
\begin{equation}\label{E:Explicit t-moduleCMPL}
  \rho_{t} = \theta I_{d} + N + E \tau
  \in \Mat_{d}(\ok[\tau]).
\end{equation}
Note that $G$ depends  only on $u_{1},\ldots,u_{r-1}$.

\begin{example}
When $r = 1$, we have
\[
\rho_{t}=C^{\otimes s_{1}}_{t} :=
\left( \begin{array}{ccccc}
\theta & 1 & & & \\
& \theta & 1 & & \\
& & \ddots & \ddots & \\
& & & \theta & 1 \\
\tau & & & & \theta
\end{array}
\right)
\in \Mat_{d}(\ok[\tau])
\]
which is called the $s_{1}$th tensor power of the Carlitz module.

When $r = 2$, we have
\[
\rho_{t} = \left( \begin{array}{c|c}
C^{\otimes(s_1 + s_2)}_{t} &
\begin{array}{ccc} 0 & \cdots & 0 \\ \vdots & & \vdots \\ -u_{1}\tau & \cdots & 0 \end{array}
\\ \hline
0 & \vphantom{ \begin{array}{ccc}
0 & \cdots & 0 \\
\vdots & & \vdots \\
-u_{1}\tau & \cdots & 0
\end{array}}
C^{\otimes s_2}_{t}
\end{array}
\right)
\in \Mat_{d}(\ok[\tau]).
\]
When $r = 3$, we have
\[
  \rho_{t} = \left(
  \begin{array}{c|c|c}
C^{\otimes(s_{1} + s_{2} + s_{3})}_{t} &
   \begin{array}{ccc}
    0 & \cdots & 0 \\
    \vdots & & \vdots \\
    -u_{1}\tau & \cdots & 0
   \end{array} &
   \begin{array}{ccc}
    0 & \cdots & 0 \\
    \vdots & & \vdots \\
    u_{1} u_{2}\tau & \cdots & 0
   \end{array}
\\ \hline
0 & C^{\otimes(s_{2} + s_{3})}_{t} &
   \begin{array}{ccc}
    0 & \cdots & 0 \\
    \vdots & & \vdots \\
    -u_{2}\tau & \cdots & 0
   \end{array}
   \\ \hline
0 & 0 &
\vphantom{   \begin{array}{ccc}
    0 & \cdots & 0 \\
    \vdots & & \vdots \\
    -u_{1}\tau & \cdots & 0
   \end{array}}
   C^{\otimes s_{3}}_{t}
  \end{array}
  \right)
  \in \Mat_{d}(\ok[\tau]).
\]
\end{example}

\subsection{Logarithm of $G$}
We denote by
\[
\log_{G} = \sum_{i \geq 0} P_{i} \tau^{i}
\]
the logarithm of the $t$-module $G$,
where we put
\[
P_{0} := I_{d}, \ \
P_{i} := \left(
\begin{array}{ccc}
P_{i}[11] & \cdots & P_{i}[1r] \\
\vdots & & \vdots \\
P_{i}[r1] & \cdots & P_{i}[rr]
\end{array}
\right)
\in \Mat_{d}(\ok), \ \
P_{i}[\ell m] \in \Mat_{d_{\ell} \times d_{m}}(\ok).
\]

\begin{proposition}\label{prop_log_coeff}
We have $P_{i}[\ell m] = 0$ for $\ell > m$.
For $\ell \leq m$, if we denote by $y_{i}[\ell m] := P_{i}[\ell m]_{d_{\ell} d_{m}}$, which is
the lower most right corner of $P_{i}[\ell m]$,
then we have
\begin{eqnarray}\label{eq_log_l=m}
y_{i}[\ell m] = \dfrac{1}{L_{i}^{d_{m}}} \ \ \ (\mathrm{if} \ \ell = m),
\end{eqnarray}
and
\begin{eqnarray}\label{eq_log_l<m}
y_{i}[\ell m] = (-1)^{m-\ell} \sum_{0 \leq i_{\ell} \leq \cdots \leq i_{m-1} < i}
\dfrac{u_{\ell}^{q^{i_{\ell}}} \cdots u_{m-1}^{q^{i_{m-1}}}}
{L_{i_{\ell}}^{s_{\ell}} \cdots L_{i_{m-1}}^{s_{m-1}} L_{i}^{d_{m}}} \ \ \ (\mathrm{if} \ \ell < m).
\end{eqnarray}
\end{proposition}

\begin{proof}
For any matrix $M = (M_{i j})$ with $M_{i j} \in \CC_{v}$, we set $M^{(s)} := (M_{i j}^{q^{s}})$ for $s\in \ZZ$.
From the functional equation

\[ \log_{G}\circ \rho_{t}=\partial \rho_{t}\circ \log_{G},  \]
by comparing the coefficient matrices we have the following identity

\[
P_{i+1} (\theta^{q^{i+1}} I_{d} + N) + P_{i} E^{(i)} = (\theta I_{d} + N) P_{i+1} \hbox{ for }i\in \ZZ_{\geq 0}.
\]
Given two square matrices of the same size $X,Y$,  we denote by $\ad(X)^{0}(Y) := Y$ and $\ad(X)^{j+1}(Y) := X (\ad(X)^{j}(Y)) - (\ad(X)^{j}(Y)) X$. From the identity above, we obtain

\[
P_{i+1} - \dfrac{\ad(N)^{1}(P_{i+1})}{\theta^{q^{i+1}}-\theta}
= -\dfrac{P_{i} E^{(i)}}{\theta^{q^{i+1}}-\theta}
\]
and hence for each $i\in \NN_{\geq 0}$,

\begin{eqnarray} \label{eq-P}
P_{i+1} = -\sum_{j=0}^{2d_{1}-2} \dfrac{\ad(N)^{j}(P_{i} E^{(i)})}{(\theta^{q^{i+1}}-\theta)^{j+1}}
\end{eqnarray}
because $\ad(N)^{2d_{1}-1}(P_{i+1}) = 0$ from the fact $N^{d_{1}}=0$. Therefore we have
\begin{equation}\label{E:P_i ell m}
P_{i}[\ell m] = 0 \hbox{ for }\ell > m
\end{equation}
by induction on $i$.

Letting
\[
Y =
\left( \begin{array}{c|ccc|c}
* & & * & & * \\ \hline
& & & & \\
* & & * & & * \\
& & & y & \\ \hline
* & & * & & *
\end{array}
\right)
\]
be a $d \times d$-matrix such that the lower most right corner of the $(\ell, m)$-th block matrix is $y$, we note that $E_{\ell}^{\rtr} Y E_{m}$ is of the form
\[
E_{\ell}^{\rtr} Y E_{m} =
\left( \begin{array}{c|ccc|c}
0 & & 0 & & 0 \\ \hline
& y & & & \\
0 & & 0 & & 0 \\
& & & & \\ \hline
0 & & 0 & & 0
\end{array}
\right)
\]
(the $d \times d$-matrix such that the upper most left corner of the $(\ell, m)$-th block matrix is $y$ and the others are zero matrices).

Since $E_{\ell}^{\rtr} N = 0$ and $\ad(N)^{j}(P_{i} E^{(i)})$ can be expressed as $\ad(N)^{j}(P_{i} E^{(i)})= N B + (-1)^{j} P_{i} E^{(i)} N^{j}$ for some $B \in \Mat_{d}(\ok)$,
we have

\begin{eqnarray} \label{eq-NPE}
E_{\ell}^{\rtr} P_{i+1} E_{m}
&=& -\sum_{j=0}^{2d_{1}-2} \dfrac{E_{\ell}^{\rtr} \ad(N)^{j}(P_{i} E^{(i)}) E_{m}}
{(\theta^{q^{i+1}}-\theta)^{j+1}} \\
&=& \sum_{j=0}^{2d_{1}-2} \dfrac{E_{\ell}^{\rtr} P_{i} E^{(i)} N^{j} E_{m}}
{(\theta-\theta^{q^{i+1}})^{j+1}} \nonumber \\
&=& \dfrac{E_{\ell}^{\rtr} P_{i} E^{(i)} N^{d_{m}-1} E_{m}} {(\theta-\theta^{q^{i+1}})^{d_{m}}}. \nonumber
\end{eqnarray}
Note that the last equality above comes from the facts that $E^{(i)} N^{j} E_{m} = 0$ if $j \neq d_{m} - 1$, and

\[
E^{(i)} N^{d_{m}-1} E_{m} =
\begin{array}{rcccccccll}
& \multicolumn{3}{c}{\overbrace{\hspace{6em}}^{\mbox{$d_{1} + \cdots + d_{m-1}$}}}
& \overbrace{\hspace{1.5em}}^{\mbox{$d_{m}$}}
& \multicolumn{3}{c}{\overbrace{\hspace{6em}}^{\mbox{$d_{m+1} + \cdots + d_{r}$}}} & & \\
\ldelim({7}{4pt}[] & & & & E[1m]^{(i)} & & & & \rdelim){7}{4pt}[] & \rdelim\}{4}{50pt}[$d_{1} + \cdots + d_{m}$] \\
& \ \ \ & \ \ \ & \ \ \ & \vdots & \ \ \ & \ \ \ & \ \ \ & & \\
& \multicolumn{3}{c}{$\mbox{\smash{\Huge $0$}}$} & E[mm]^{(i)} & \multicolumn{3}{c}{$\mbox{\smash{\Huge $0$}}$} & & \\
& & & & & & & & & \\
& & & & & & & & & \\
& & & & $\mbox{\smash{\Huge $0$}}$ & & & & &
\end{array}
\] if $j = d_{m} - 1$. Note that we also have
\[
E_{\ell}^{\rtr} P_{i} =
\begin{array}{rcccccccll}
& & \stackrel{d_{1}th}{\downarrow} & & \stackrel{(d_{1}+d_{2})th}{\downarrow} & & & \stackrel{(d)th}{\downarrow} & & \\
\ldelim({8}{4pt}[] & & & & & & &  & \rdelim){8}{4pt}[] & \\
& \multicolumn{2}{c}{$\mbox{\smash{\Huge $0$}}$} & \cdots & \cdots & \cdots & \multicolumn{2}{c}{$\mbox{\smash{\Huge $0$}}$} & & \\
& & & & & & & & & \\
& * & y_{i}[\ell 1] & * & y_{i}[\ell 2] & \cdots & * & y_{i}[\ell r] & &
\leftarrow (\ d_{1} + \cdots + d_{\ell-1} + 1)th \\
& & & & & & & & & \\
& & & & & & & & & \\
& \multicolumn{2}{c}{$\mbox{\smash{\Huge $0$}}$} & \cdots & \cdots & \cdots & \multicolumn{2}{c}{$\mbox{\smash{\Huge $0$}}$} & & \\[10.0pt]
\end{array}
\]

By comparing with the upper most left corner of the $(\ell, m)$-th block matrix (which is the $(d_{1} + \cdots + d_{\ell-1} + 1, d_{1} + \cdots + d_{m-1} + 1)$-th entry)
of both sides of the equation (\ref{eq-NPE}), we have from (\ref{E:P_i ell m}) that

\begin{eqnarray}\label{eq_rec_l=m}
L_{i+1}^{d_{m}} y_{i+1}[\ell m] = L_{i}^{d_{m}} y_{i}[\ell m] \ \ \ (\mathrm{if} \ \ell = m)
\end{eqnarray}
and that
\begin{eqnarray}\label{eq_rec_l<m}
L_{i+1}^{d_{m}} y_{i+1}[\ell m] =
L_{i}^{d_{m}} \sum_{n=\ell}^{m-1} y_{i}[\ell n] (-1)^{m-n} \prod_{e=n}^{m-1} u_{e}^{q^{i}}
+ L_{i}^{d_{m}} y_{i}[\ell m] \ \ \  (\mathrm{if} \ \ell < m)
\end{eqnarray}
(note that $L_{i+1} = (\theta-\theta^{q^{i+1}}) L_{i}$).
By definition, we have $y_{0}[\ell m] = 1$ if $\ell = m$ and $y_{0}[\ell m] = 0$ if $\ell < m$.
We fix $\ell$ and show the equalities (\ref{eq_log_l=m}) and (\ref{eq_log_l<m}) hold
by the induction on $m$ ($\geq \ell$).

When $m = \ell$, we can show this easily by the recurrence relation (\ref{eq_rec_l=m}).
Let $m > \ell$ and assume that the equality (\ref{eq_log_l<m}) is true for $y_{i}[\ell n]$
with $\ell \leq n < m$ and $i \geq 0$.
By the recurrence relation (\ref{eq_rec_l<m}), we have

\begin{eqnarray*}
\begin{split}
L_{i+1}^{d_{m}} y_{i+1}[\ell m]
&= L_{i}^{d_{m}}
\sum_{n=\ell+1}^{m-1} (-1)^{n-\ell} \sum_{0 \leq i_{\ell} \leq \cdots \leq i_{n-1} < i}
\dfrac{u_{\ell}^{q^{i_{\ell}}} \cdots u_{n-1}^{q^{i_{n-1}}}}
{L_{i_{\ell}}^{s_{\ell}} \cdots L_{i_{n-1}}^{s_{n-1}} L_{i}^{d_{n}}}
(-1)^{m-n}\prod_{e=n}^{m-1} u_{e}^{q^{i}} \\
&\quad + L_{i}^{d_{m}} \dfrac{1}{L_{i}^{d_{\ell}}}(-1)^{m-\ell} \prod_{e=\ell}^{m-1} u_{e}^{q^{i}}
+ L_{i}^{d_{m}} y_{i}[\ell m] \\
&=
(-1)^{m-\ell} \sum_{n=\ell+1}^{m-1} \ \sum_{0 \leq i_{\ell} \leq \cdots \leq i_{n-1} < i}
\dfrac{u_{\ell}^{q^{i_{\ell}}} \cdots u_{n-1}^{q^{i_{n-1}}} u_{n}^{q^{i}} \cdots u_{m-1}^{q^{i}}}
{L_{i_{\ell}}^{s_{\ell}} \cdots L_{i_{n-1}}^{s_{n-1}} L_{i}^{s_{n}} \cdots L_{i}^{s_{m-1}}} \\
&\quad + (-1)^{m-\ell}
\dfrac{u_{\ell}^{q^{i}} \cdots u_{m-1}^{q^{i}}}{L_{i}^{s_{\ell}} \cdots L_{i}^{s_{m-1}}}
+ L_{i}^{d_{m}} y_{i}[\ell m] \\
&= (-1)^{m-\ell} \sum_{\substack{0 \leq i_{\ell} \leq \cdots \leq i_{m-1} \\ i_{m-1} = i}}
\dfrac{u_{\ell}^{q^{i_{\ell}}} \cdots u_{m-1}^{q^{i_{m-1}}}}
{L_{i_{\ell}}^{s_{\ell}} \cdots L_{i_{m-1}}^{s_{m-1}}}
+ L_{i}^{d_{m}} y_{i}[\ell m].
\end{split}
\end{eqnarray*}
Since $y_{0}[\ell m] = 0$, we have

\begin{eqnarray*}
L_{i}^{d_{m}} y_{i}[\ell m]
&=& (-1)^{m-\ell} \sum_{i' = 0}^{i-1} \ \sum_{\substack{0 \leq i_{\ell} \leq \cdots \leq i_{m-1} \\ i_{m-1} = i'}}
\dfrac{u_{\ell}^{q^{i_{\ell}}} \cdots u_{m-1}^{q^{i_{m-1}}}}
{L_{i_{\ell}}^{s_{\ell}} \cdots L_{i_{m-1}}^{s_{m-1}}} \\
&=& (-1)^{m-\ell} \sum_{0 \leq i_{\ell} \leq \cdots \leq i_{m-1} < i}
\dfrac{u_{\ell}^{q^{i_{\ell}}} \cdots u_{m-1}^{q^{i_{m-1}}}}{L_{i_{\ell}}^{s_{\ell}} \cdots L_{i_{m-1}}^{s_{m-1}}}
\end{eqnarray*}
and hence the equalities (\ref{eq_log_l=m}) and (\ref{eq_log_l<m}) are true
for each $m \geq \ell$ by induction.
\end{proof}

\subsection{$v$-adic convergence of CMPLs} \label{subsection-log-convergence}
Next, we consider when the sum $\log_{G} \bx$ converges in
$\Lie G(\CC_{v}) = \CC_{v}^{d}$ for $\bx \in G(\CC_{v}) = \CC_{v}^{d}$.
For any matrix $M = (M_{i j})$ with $M_{i j} \in \CC_{v}$,
we denote $|M|_{v} := \max \{ |M_{i j}|_{v} \}$.
Assume that $|u_m|_{v} \leq 1$ for each $1 \leq m <r$.
Then $\log_{G} \bx$ converges for each $\bx \in G(\CC_{v})$ with $|\bx|_{v} < 1$.
Indeed, it is clear that $|N|_{v} = 1$ and $|E^{(i)}|_{v} \leq 1$.
Since $v^2$ does not divide $\theta^{q^{i+1}} - \theta$ for each $i \geq 0$,
we have $1 \geq |\theta^{q^{i+1}} - \theta|_{v} \geq |v|_{v}$.
Thus by the equation (\ref{eq-P}), we obtain
\[
|P_{i+1}|_{v} \leq |P_{i}|_{v} |\theta^{q^{i+1}} - \theta|_{v}^{-(2d_{1} - 2+1)}
\leq |P_{i}|_{v} |v|_{v}^{-(2d_{1}-1)}.
\]
Since $P_{0} = I_{d}$, we have an upper bound
\[
|P_{i}|_{v} \leq |v|_{v}^{-i(2d_{1}-1)}
\]
for each $i \geq 0$.
Therefore if $|\bx|_{v} < 1$, we have
\[
|P_{i} \bx^{(i)}|_{v} \leq |P_{i}|_{v} |\bx|_{v}^{q^{i}} \leq |v|_{v}^{-i(2d_{1}-1)} |\bx|_{v}^{q^{i}} \to 0 \ \ (i \to \infty),
\]
and hence the sum converges.

Let
\begin{equation}\label{E:v_s,u}
\bv = \bv_{\fs, \bu} :=
\begin{array}{rcll}
\ldelim( {15}{4pt}[] & 0 & \rdelim) {15}{4pt}[] & \rdelim\}{4}{10pt}[$d_{1}$] \\
& \vdots & & \\
& 0 & & \\
& (-1)^{r-1} u_{1} \cdots u_{r} & & \\
& 0 & & \rdelim\}{4}{10pt}[$d_{2}$] \\
& \vdots & & \\
& 0 & & \\
& (-1)^{r-2} u_{2} \cdots u_{r} & & \\
& \vdots & & \vdots \\
& 0 & & \rdelim\}{4}{10pt}[$d_{r}$] \\
& \vdots & & \\
& 0 & & \\
& u_{r} & & \\[10pt]
\end{array} \in G(\ok).
\end{equation}
It is clear that $|\bv|_{v} < 1$ when $|u_{m}|_{v} \leq 1$ for each $1 \leq m < r$ and
$|u_{r}|_{v} < 1$.
Thus in this case, the sum $\log_{G} \bv$ converges $v$-adically.

\begin{remark}
The authors do not know  the precise $v$-adic convergence domain of $\log_{G}$.
\end{remark}

\begin{theorem}\label{eq_log_value}
Let $u_{1}, \dots, u_{r} \in \ok^{\times}$ with $|u_{m}|_{v} \leq 1$ for each $1 \leq m < r$ and $|u_{r}|_{v} < 1$.
Let $G$ and $\bv$ be as above.
Then we have
\[
\log_{G} \bv =
\begin{array}{rcll}
\ldelim( {15}{4pt}[] & * & \rdelim) {15}{4pt}[] & \rdelim\}{4}{10pt}[$d_{1}$] \\
& \vdots & & \\
& * & & \\
& (-1)^{r-1}\Li_{(s_{r}, \dots, s_{1})}^{\star}(u_{r}, \dots, u_{1})_{v} & & \\
& * & & \rdelim\}{4}{10pt}[$d_{2}$] \\
& \vdots & & \\
& * & & \\
& (-1)^{r-2}\Li_{(s_{r}, \dots, s_{2})}^{\star}(u_{r}, \dots, u_{2})_{v} & & \\
& \vdots & & \vdots \\
& * & & \rdelim\}{4}{10pt}[$d_{r}$] \\
& \vdots & & \\
& * & & \\
& \Li^{\star}_{s_{r}}(u_{r})_{v} & & \\[10pt]
\end{array} \ \ \ \in \Lie G(\CC_{v}).
\]
\end{theorem}

\begin{proof}
Let $1 \leq \ell \leq r$.
By Proposition $\ref{prop_log_coeff}$ and the definition of $\bv$,
the last coordinate of the $\ell$-th block ($=$ the $(d_{1} + \cdots + d_{\ell})$-th component)
of the vector $\log_{G} \bv$ is
\begin{eqnarray*}
\begin{split}
& \sum_{i \geq 0} \sum_{m=\ell}^{r} y_{i}[\ell m] (-1)^{r-m}u_{m}^{q^{i}} \cdots u_{r}^{q^{i}} \\
&= \sum_{i \geq 0} \left( \dfrac{1}{L_{i}^{d_{\ell}}} (-1)^{r-\ell} u_{\ell}^{q^{i}} \cdots u_{r}^{q^{i}}
\right. \\
&\quad + \left. \sum_{m=\ell+1}^{r} (-1)^{m-\ell} \sum_{0 \leq i_{\ell} \leq \cdots \leq i_{m-1} < i}
\dfrac{u_{\ell}^{q^{i_{\ell}}} \cdots u_{m-1}^{q^{i_{m-1}}}}
{L_{i_{\ell}}^{s_{\ell}} \cdots L_{i_{m-1}}^{s_{m-1}} L_{i}^{d_{m}}}
(-1)^{r-m}u_{m}^{q^{i}} \cdots u_{r}^{q^{i}} \right) \\
&= (-1)^{r-\ell} \sum_{i \geq 0} \left(
\dfrac{u_{\ell}^{q^{i}} \cdots u_{r}^{q^{i}}}{L_{i}^{s_{\ell}} \cdots L_{i}^{s_{r}}}
+ \sum_{m=\ell+1}^{r} \ \sum_{0 \leq i_{\ell} \leq \cdots \leq i_{m-1} < i}
\dfrac{u_{\ell}^{q^{i_{\ell}}} \cdots u_{m-1}^{q^{i_{m-1}}} u_{m}^{q^{i}} \cdots u_{r}^{q^{i}}}
{L_{i_{\ell}}^{s_{\ell}} \cdots L_{i_{m-1}}^{s_{m-1}} L_{i}^{s_{m}} \cdots L_{i}^{s_{r}}}
\right) \\
&= (-1)^{r-\ell} \sum_{0 \leq i_{\ell} \leq \cdots \leq i_{r}}
\dfrac{u_{\ell}^{q^{i_{\ell}}} \cdots u_{r}^{q^{i_{r}}}}
{L_{i_{\ell}}^{s_{\ell}} \cdots L_{i_{r}}^{s_{r}}} \\
&= (-1)^{r-\ell} \Li_{(s_{r}, \dots, s_{\ell})}^{\star}(u_{r}, \dots, u_{\ell})_{v},
\end{split}
\end{eqnarray*}
where the identities hold $v$-adically.
\end{proof}

\begin{remark}
Note that one has the same identity in the theorem above in the $\infty$-adic setting once we put suitable restrictions on the $\infty$-adic valuations of $u_i$ for which the series in question are defined.
\end{remark}

\begin{remark}\label{Rem:G and v}
The $t$-module $G$ and algebraic point $\bv$ defined above are exactly identified with $\Ext_{\cF}^{1}\left({\bf{1}},M' \right)$ and $M$ respectively in \cite[Thm.~4.3.2]{CPY14} through \cite[Thms.~5.2.1, 5.2.3]{CPY14}.
\end{remark}

\begin{remark}
If we replace $u_{1},\ldots,u_{r}$ by $r$ independent variables $t_{1},\ldots,t_{r}$, then the $t$-module $G$ is defined over $A[t_{1},\ldots,t_{r}]$ and the formula of $\log_{G}{\bf{v}}$ in Theorem~\ref{eq_log_value} is still valid as power series in the variables $t_{1},\ldots,t_{r}$ (when ignoring convergence).

\end{remark}

\section{Analytic continuation}\label{Sec:Analytic continuation}
To simplify the notation, we denote by $[-]$ the $\FF_{q}[t]$-actions on any $t$-modules.
For a monic irreducible polynomial $v \in A$, put $v(t):=v|_{\theta=t}\in \FF_{q}[t]$.
Let $\mathcal{O}_{\CC_{v}} := \{ \alpha \in \CC_{v} ; |\alpha|_{v} \leq 1 \}$ and $\fm_{v}$ be the maximal ideal of $\mathcal{O}_{\CC_{v}}$.
In $\S$ \ref{subsection-CMPLs-CMSPLs}, we see that $\Li_{\fs}$ and $\Li_{\fs}^{\star}$ converge $v$-adically on the domain
\[
\{ (z_{1}, \dots, z_{r}) | z_{1} \in \fm_{v} \ \mathrm{and} \ z_{i} \in \mathcal{O}_{\CC_{v}} \ \mathrm{for \ each} \ 2 \leq i \leq r \}.
\]
The purpose of this section is to study analytic continuation for $\Li_{\fs}$ and $\Li_{\fs}^{\star}$. Precisely, the $v$-adic convergence domain of them can be extended to $\mathcal{O}_{\CC_{v}}^{r}$.

\subsection{Analytic continuation of CMSPLs}

\begin{proposition}\label{proposition-log-converge}
Let $u_{1}, \dots, u_{r} \in \ok^{\times}$ with $|u_{i}|_{v} \leq 1$ (i.e., $u_{i} \in \mathcal{O}_{\CC_{v}}$) for each $1 \leq i \leq r$. Let $G$ be the $t$-module defined in (\ref{E:Explicit t-moduleCMPL}) and $\bv\in G(\ok)$ be defined in (\ref{E:v_s,u}). Let $\ell \geq 1$ be an integer such that each image of $u_{i}$ in $\mathcal{O}_{\CC_{v}} / \fm_{v} \cong \overline{\FF_{q}}$ is contained in $\FF_{q^{\ell}}$.
Then
\[ \log_{G} \left([v(t)^{d_{1} \ell} - 1] [v(t)^{d_{2} \ell} - 1] \cdots [v(t)^{d_{r} \ell} - 1] (\bv)\right) \]
converges in $\Lie G(\CC_{v})$.
\end{proposition}

\begin{proof}
For $1 \leq i \leq r$, let $G_{i}$ be the $t$-module whose $t$-action is defined by
the square matrix of size $d_{1} + \cdots + d_{i}$ cut from the upper left square of
the matrix defined in (\ref{E:Explicit t-moduleCMPL}), and note that it is defined over $\ok\cap \mathcal{O}_{\CC_{v}}$. We also set $G_{0} := 0$.
Thus we have the exact sequence
\[
\xymatrix{
0 \ar[r] & G_{i-1} \ar[r] & G_{i} \ar[r]^{\pi_{i}} & \bC^{\otimes d_{i}} \ar[r] & 0
}
\]
for each $1 \leq i \leq r$. Note that $G_{r}=G$.
We denote by $\overline{\pi}_{i}$ the induced map on the $\overline{\FF_{q}}$-valued points.

Denote by $\overline{\bv}_{r}$ the image of $\bv$ in $G_{r}(\overline{\FF_{q}})$.
Since by \cite[Prop.~1.6.1]{AT90} we have $[v(t)^{d_{r}}] = \tau^{\deg v}$ on $\bC^{\otimes d_{r}}(\overline{\FF_{q}})$,
it follows that $\overline{\pi}_{r} ([v(t)^{d_{r} \ell} - 1] (\overline{\bv}_{r})) = [v(t)^{d_{r} \ell} - 1] (\overline{\pi}_{r} (\overline{\bv}_{r})) = 0$.
Thus
\[
\overline{\bv}_{r-1} := [v(t)^{d_{r} \ell} - 1] (\overline{\bv}_{r}) \in \Ker \overline{\pi}_{r} = G_{r-1}(\overline{\FF_{q}}).
\]
Similarly, we have
\begin{eqnarray*}
\overline{\bv}_{r-2} &:=& [v(t)^{d_{r-1} \ell} - 1] (\overline{\bv}_{r-1}) \in \Ker \overline{\pi}_{r-1} = G_{r-2}(\overline{\FF_{q}}) \\
&\vdots& \\
\overline{\bv}_{1} &:=& [v(t)^{d_{2} \ell} - 1] (\overline{\bv}_{2}) \in \Ker \overline{\pi}_{2} = G_{1}(\overline{\FF_{q}}) \\
\overline{\bv}_{0} &:=& [v(t)^{d_{1} \ell} - 1] (\overline{\bv}_{1}) \in \Ker \overline{\pi}_{1} = G_{0}(\overline{\FF_{q}}). \\
\end{eqnarray*}
It follows that $\overline{\bv}_{0} = [v(t)^{d_{1} \ell} - 1] [v(t)^{d_{2} \ell} - 1] \cdots [v(t)^{d_{r} \ell} - 1] (\overline{\bv}_{r}) = 0$ in $G(\overline{\FF}_{q})$ and so $[v(t)^{d_{1} \ell} - 1] [v(t)^{d_{2} \ell} - 1] \cdots [v(t)^{d_{r} \ell} - 1] (\bv) \in G(\fm_{v})$, at which $\log_{G}$ converges $v$-adically (see $\S$ \ref{subsection-log-convergence}).
\end{proof}

\begin{definition}
Let $(u_{1}, \dots, u_{r}) \in \mathcal{O}_{\CC_{v}}^{r}$ and let $a \in \FF_{q}[t]$ be nonzero such that $\log_{G}\left([a] (\bv)\right)$ converges in $\Lie G(\CC_{v})$. We define $\Li_{(s_{r}, \dots, s_{1})}^{\star}(u_{r}, \dots, u_{1})_{v}$ to be the value
\[
\dfrac{(-1)^{r-1}}{a(\theta)} \times {\textit the \ } d_{1} {\textit -th \ coordinate \ of \ }
\log_{G}\left([a] (\bv)\right).
\]
\end{definition}

\begin{remark} \label{remark-analytic-continuation-CMSPLs}
(1)
By Theorem \ref{eq_log_value}, it is clear that the definition above is independent of the choices of $a$ and $\Li_{(s_{r}, \dots, s_{1})}^{\star}(u_{r}, \dots, u_{1})_{v}$ coincides with the definition in (\ref{E:star-CMPL}) when $|u_{r}|_{v} < 1$.

(2)
By Proposition \ref{proposition-log-converge}, $\Li_{(s_{r}, \dots, s_{1})}^{\star}(u_{r}, \dots, u_{1})_{v}$ is
\[
\dfrac{(-1)^{r-1}}{(v^{d_{1} \ell}-1) \cdots (v^{d_{r} \ell}-1)} \times {\textit the \ } d_{1} {\textit -th \ coordinate \ of \ }
\log_{G}\left([v(t)^{d_{1} \ell}-1] \cdots [v(t)^{d_{r} \ell}-1] (\bv)\right).
\]

(3) Given $(u_{1}, \dots, u_{r}) \in \mathcal{O}_{\CC_{v}}^{r}$ and nonzero $a \in \FF_{q}[t]$ such that $\log_{G}\left([a] (\bv)\right)$ converges in $\Lie G(\CC_{v})$, for each $1\leq \ell \leq r$ the ($d_{1} + d_{2} + \cdots + d_{\ell}$)-th component of $\log_{G}\left([a] (\bv)\right)$ is
\[ (-1)^{r-\ell} a(\theta) \Li_{(s_{r}, \dots, s_{\ell})}^{\star}(u_{r}, \dots, u_{\ell})_{v}.\]
\end{remark}

\subsection{Analytic continuation of CMPLs}

\begin{lemma}\label{lem-CMPL-CMSPL} For any $s_{1},\ldots,s_{r}\in \NN$, we have the following identity as power series.
\begin{eqnarray*}
& & \Li^{\star}_{(s_{r}, \ldots, s_{1})}(z_{r},\ldots,z_{1}) \\
&=& \sum_{\ell=2}^{r} (-1)^{\ell} \Li_{(s_{1}, \ldots, s_{\ell-1})}(z_{1}, \ldots, z_{\ell-1}) \Li^{\star}_{(s_{r}, \ldots, s_{\ell})}(z_{r}, \ldots, z_{\ell})
+ (-1)^{r+1} \Li_{(s_{1}, \ldots, s_{r})}(z_{1}, \ldots, z_{r})
\end{eqnarray*}
\end{lemma}

\begin{proof}
\begin{eqnarray*}
& & {\rm LHS} \ = \ \sum_{i_{1}} \frac{ z_{1}^{ q^{ i_{1} } } }{L_{ i_{1} }^{ s_{1} }}
\sum_{i_{r} \geq \cdots \geq i_{2}} \frac{z_{r}^{q^{i_{r}}}\cdots z_{2}^{q^{i_{2}}}}{L_{i_{r}}^{s_{r}}\cdots L_{i_{2}}^{s_{2}}}
- \sum_{\substack{ i_{r} \geq \cdots \geq i_{2} \\ i_{1} > i_{2} }} \frac{z_{r}^{q^{i_{r}}}\cdots z_{1}^{q^{i_{1}}}} {L_{i_{r}}^{s_{r}}\cdots L_{i_{1}}^{s_{1}}} \\
&=& \sum_{i_{1}} \frac{ z_{1}^{ q^{ i_{1} } } }{L_{ i_{1} }^{ s_{1} }}
\sum_{i_{r} \geq \cdots \geq i_{2}} \frac{z_{r}^{q^{i_{r}}}\cdots z_{2}^{q^{i_{2}}}}{L_{i_{r}}^{s_{r}}\cdots L_{i_{2}}^{s_{2}}}
- \sum_{i_{1} > i_{2}} \frac{z_{1}^{q^{i_{1}}} z_{2}^{q^{i_{2}}}} {L_{i_{1}}^{s_{1}} L_{i_{2}}^{s_{2}}}
\sum_{i_{r} \geq \cdots \geq i_{3}} \frac{z_{r}^{q^{i_{r}}}\cdots z_{3}^{q^{i_{3}}}} {L_{i_{r}}^{s_{r}}\cdots L_{i_{3}}^{s_{3}}}
+ \sum_{\substack{ i_{r} \geq \cdots \geq i_{3} \\ i_{1} > i_{2} > i_{3}}} \frac{z_{r}^{q^{i_{r}}}\cdots z_{1}^{q^{i_{1}}}} {L_{i_{r}}^{s_{r}}\cdots L_{i_{1}}^{s_{1}}} \\
&=& \sum_{i_{1}} \frac{ z_{1}^{ q^{ i_{1} } } }{L_{ i_{1} }^{ s_{1} }}
\sum_{i_{r} \geq \cdots \geq i_{2}} \frac{z_{r}^{q^{i_{r}}}\cdots z_{2}^{q^{i_{2}}}}{L_{i_{r}}^{s_{r}}\cdots L_{i_{2}}^{s_{2}}}
- \sum_{i_{1} > i_{2}} \frac{z_{1}^{q^{i_{1}}} z_{2}^{q^{i_{2}}}} {L_{i_{1}}^{s_{1}} L_{i_{2}}^{s_{2}}}
\sum_{i_{r} \geq \cdots \geq i_{3}} \frac{z_{r}^{q^{i_{r}}}\cdots z_{3}^{q^{i_{3}}}} {L_{i_{r}}^{s_{r}}\cdots L_{i_{3}}^{s_{3}}} \\
& & +  \sum_{i_{1} > i_{2} > i_{3}} \frac{z_{1}^{q^{i_{1}}} z_{2}^{q^{i_{2}}} z_{3}^{q^{i_{3}}}} {L_{i_{1}}^{s_{1}} L_{i_{2}}^{s_{2}} L_{i_{3}}^{s_{3}}}
\sum_{i_{r} \geq \cdots \geq i_{4}} \frac{z_{r}^{q^{i_{r}}} \cdots z_{4}^{q^{i_{4}}}} {L_{i_{r}}^{s_{r}}\cdots L_{i_{4}}^{s_{4}}}
- \sum_{\substack{ i_{r} \geq \cdots \geq i_{4} \\ i_{1} > i_{2} > i_{3} > i_{4}}} \frac{z_{r}^{q^{i_{r}}}\cdots z_{1}^{q^{i_{1}}}} {L_{i_{r}}^{s_{r}}\cdots L_{i_{1}}^{s_{1}}} \\
& \vdots & \\
&=& \sum_{i_{1}} \frac{ z_{1}^{ q^{ i_{1} } } }{L_{ i_{1} }^{ s_{1} }}
\sum_{i_{r} \geq \cdots \geq i_{2}} \frac{z_{r}^{q^{i_{r}}}\cdots z_{2}^{q^{i_{2}}}}{L_{i_{r}}^{s_{r}}\cdots L_{i_{2}}^{s_{2}}}
- \sum_{i_{1} > i_{2}} \frac{z_{1}^{q^{i_{1}}} z_{2}^{q^{i_{2}}}} {L_{i_{1}}^{s_{1}} L_{i_{2}}^{s_{2}}}
\sum_{i_{r} \geq \cdots \geq i_{3}} \frac{z_{r}^{q^{i_{r}}}\cdots z_{3}^{q^{i_{3}}}} {L_{i_{r}}^{s_{r}}\cdots L_{i_{3}}^{s_{3}}} + \cdots \\
& & \cdots +
(-1)^{r-1} \sum_{i_{1} > \cdots > i_{r-2}} \frac{z_{1}^{q^{i_{1}}} \cdots z_{r-2}^{q^{i_{r-2}}}} {L_{i_{1}}^{s_{1}} \cdots L_{i_{r-2}}^{s_{r-2}}}
\sum_{i_{r} \geq i_{r-1}} \frac{z_{r}^{q^{i_{r}}} z_{r-1}^{q^{i_{r-1}}}} {L_{i_{r}}^{s_{r}} L_{i_{r-1}}^{s_{r-1}}}
+ (-1)^{r} \sum_{\substack{ i_{r} \geq i_{r-1} \\ i_{1} > \cdots > i_{r-1}}} \frac{z_{r}^{q^{i_{r}}}\cdots z_{1}^{q^{i_{1}}}} {L_{i_{r}}^{s_{r}}\cdots L_{i_{1}}^{s_{1}}} \\
&=& \sum_{i_{1}} \frac{ z_{1}^{ q^{ i_{1} } } }{L_{ i_{1} }^{ s_{1} }}
\sum_{i_{r} \geq \cdots \geq i_{2}} \frac{z_{r}^{q^{i_{r}}}\cdots z_{2}^{q^{i_{2}}}}{L_{i_{r}}^{s_{r}}\cdots L_{i_{2}}^{s_{2}}}
- \sum_{i_{1} > i_{2}} \frac{z_{1}^{q^{i_{1}}} z_{2}^{q^{i_{2}}}} {L_{i_{1}}^{s_{1}} L_{i_{2}}^{s_{2}}}
\sum_{i_{r} \geq \cdots \geq i_{3}} \frac{z_{r}^{q^{i_{r}}}\cdots z_{3}^{q^{i_{3}}}} {L_{i_{r}}^{s_{r}}\cdots L_{i_{3}}^{s_{3}}} + \cdots \\
& & \cdots +
(-1)^{r-1} \sum_{i_{1} > \cdots > i_{r-2}} \frac{z_{1}^{q^{i_{1}}} \cdots z_{r-2}^{q^{i_{r-2}}}} {L_{i_{1}}^{s_{1}} \cdots L_{i_{r-2}}^{s_{r-2}}}
\sum_{i_{r} \geq i_{r-1}} \frac{z_{r}^{q^{i_{r}}} z_{r-1}^{q^{i_{r-1}}}} {L_{i_{r}}^{s_{r}} L_{i_{r-1}}^{s_{r-1}}} \\
& & + (-1)^{r} \sum_{i_{1} > \cdots > i_{r-1}} \frac{z_{1}^{q^{i_{1}}} \cdots z_{r-1}^{q^{i_{r-1}}}} {L_{i_{1}}^{s_{1}} \cdots L_{i_{r-1}}^{s_{r-1}}}
\sum_{i_{r}} \frac{z_{r}^{q^{i_{r}}}} {L_{i_{r}}^{s_{r}}}
+ (-1)^{r+1} \sum_{i_{1} > \cdots > i_{r}} \frac{z_{r}^{q^{i_{r}}}\cdots z_{1}^{q^{i_{1}}}} {L_{i_{r}}^{s_{r}}\cdots L_{i_{1}}^{s_{1}}} \\
&=& {\rm RHS}
\end{eqnarray*}
\end{proof}

\begin{remark}
To see the analogue of the identity above, we refer the reader to \cite[Prop.~6]{IKOO11}, \cite[Thm.~2.13]{SS15} and \cite[Thm.~3]{Z05}.
\end{remark}

\begin{definition} \label{analytic-continuation-CMPLs}
Let $(u_{1}, \dots, u_{r}) \in \mathcal{O}_{\CC_{v}}^{r}$ and let $a \in \FF_{q}[t]$ be nonzero such that $\log_{G}\left([a] (\bv)\right)$ converges in $\Lie G(\CC_{v})$. We define $\Li_{(s_{1}, \dots, s_{r})}(u_{1}, \dots, u_{r})_{v}$ inductively by the following identity
\begin{eqnarray*}
\Li_{(s_{1}, \ldots, s_{r})}(u_{1}, \ldots, u_{r})_{v}
&:=& (-1)^{r+1} \Li^{\star}_{(s_{r}, \ldots, s_{1})}(u_{r},\ldots,u_{1})_{v} \\
& &+ \sum_{\ell=2}^{r} (-1)^{r+\ell} \Li_{(s_{1}, \ldots, s_{\ell-1})}(u_{1}, \ldots, u_{\ell-1})_{v} \Li^{\star}_{(s_{r}, \ldots, s_{\ell})}(u_{r}, \ldots, u_{\ell})_{v}
\end{eqnarray*}
\end{definition}

\begin{remark}\label{Rmk:analytic-continuation-CMPLs}
The methods in Proposition~\ref{proposition-log-converge} are inspired by comments of one referee for $r=1$, and we are grateful to him or her to share the ideas with us. Note that the methods presented above only enable us to extend the $v$-adic convergence domain of $\Li_{\fs}$ to $\mathcal{O}_{\CC_{v}}^{r}$.
For example, we simply consider $\fs = (1)$ and in this case $\Li_{\fs}(z) = \log_{C}(z)$, which is the Carlitz logarithm convergent $v$-adically on $|z|_{v} < 1$. We take $v = \theta$ and $z \in \CC_{v}$ for which $|z|_{v} > 1$.
Then one sees that $|C_{t^{n}-1}(z)|_{v} = |z^{q^{n}}|_{v} > 1$ and so $\log_{C}(C_{t^{n}-1}(z))$ does not converge $v$-adically for each $n \in \NN$.
\end{remark}


\section{$v$-adic vanishing principle for CMPLs}
In what follows, we fix a monic prime $v$ of $A$ and an $r$-tuple $\fs=(s_{1},\ldots,s_{r})\in \NN^{r}$. We further fix $\bu=(u_{1},\ldots,u_r)\in (\ok^{\times})^{r}\cap \mathcal{O}_{\CC_{v}}^{r} $ and so
\[ \Li_{(s_{\ell}, \dots, s_{r})}(u_{\ell}, \dots, u_{r})_{v} \ \mathrm{and} \ \Li_{(s_{r}, \dots, s_{\ell})}^{\star}(u_{r}, \dots, u_{\ell})_{v} \]
are defined for $\ell=1,\ldots,r$.

\subsection{The vanishing principle}
We continue the notation and hypotheses above.
We let $G = (\GG_{a}^{d}, \rho)$ be the $t$-module defined in (\ref{E:Explicit t-moduleCMPL}).
Define $\bv=\bv_{\fs,\bu}\in G(\ok)$ to be the algebraic point given  in (\ref{E:v_s,u}). Although we have the functional identity for the logarithm of a $t$-module~\ref{E:FunEquaLog}, we can not evaluate at arbitrary points as the logarithm in question is not an entire function. The following lemma provides  the situation fitting into our need for later study on $\log_{G}$.

\begin{lemma}\label{lem_log_commute}
Let $H = (\GG_{a}^{n}, \sigma)$ be a $t$-module defined over  $\mathcal{O}_{\CC_{v}}$
and $\bx \in H(\CC_{v})$ be a point such that $|\bx|_{v} < 1$.
Then for each $a \in \FF_{q}[t]$, we have
$\log_{H}(\sigma_{a}(\bx)) = \partial \sigma_{a}(\log_{H} (\bx))$.
\end{lemma}

\begin{proof}
Let $X$ be a new variable and consider a map
$\log_{H, X} \colon (X \CC_{v}[\![X]\!])^{n} \to (X \CC_{v}[\![X]\!])^{n}$ by
\[
\log_{H, X} (\bg) := \sum_{i=0}^{\infty} Q_{i} \bg^{(i)},
\]
where $Q_{i} \in \Mat_{n}(\CC_{v})$ are the coefficient matrices of $\log_{H}$ and $\bg^{(i)}$ is obtained by raising each component of $\bg$ to the $q^{i}$th power. Then we have $\log_{H, X} (\sigma_{a} (\bg)) = \partial \sigma_{a} (\log_{H, X} (\bg))$
for each $a \in \FF_{q}[t]$ and $\bg \in (X \CC_{v}[\![X]\!])^{n}$. Let $\CC_{v} \{X\}$ be the ring consisting of all power series which converge on $|X|_{v} \leq 1$, and we call $\CC_{v} \{X\}$ the Tate algebra over $\CC_{v}$.

Let $\fm_{v}$ be the maximal ideal of $\mathcal{O}_{\CC_{v}}$.
For each $\bx \in \fm_{v}^{n}$ and $i \geq 1$, it is clear that
$\log_{H, X} (X^{i} \bx) \in \CC_{v} \{X\}^{n}$ and its value at $X = 1$ is $\log_{H} \bx$.
Since $\log_{H, X}$, $\log_{H}$ and the evaluation map $X \mapsto 1$ are additive,
we can extend these operations from $\{ X^{i} \bx \}$ to $(X \fm_{v}[X])^{n}$,
and hence we have the commutative diagram
\[
\xymatrix{
(X \fm_{v}[X])^{n} \ar[r] \ar[d]_{\log_{H, X}} & \fm_{v}^{n} \ar[d]^{\log_{H}} \\
\CC_{v} \{ X \}^{n} \ar[r] & \CC_{v}^{n} = \Lie H
}
\]
where the horizontal maps are the evaluating map $\bg(X) \mapsto \bg(1)$.

Let $\bx \in H(\CC_{v})$ be a point such that $|\bx|_{v} < 1$.
Note that $\log_{H}(\sigma_{a}(\bx))$ converges since $|\sigma_{a} (\bx)|_{v} \leq |\bx|_{v} < 1$.
From the fact that $\sigma_{a}(X \bx) \in (X \fm_{v}[X])^{n}$, we have
\begin{eqnarray*}
\log_{H}(\sigma_{a}(\bx)) = \log_{H}(\sigma_{a}(X \bx)|_{X=1})
= \log_{H, X}(\sigma_{a}(X \bx))|_{X=1} \\
= \partial \sigma_{a}(\log_{H, X}(X \bx))|_{X=1}
= \partial \sigma_{a}(\log_{H} (\bx)).
\end{eqnarray*}
\end{proof}

\begin{theorem}\label{T:VanishingCriterion}
Given $(s_{1}, \dots, s_{r}) \in \NN^{r}$, let $v \in A$ be a monic irreducible polynomial and $(u_{1}, \dots, u_{r}) \in (\ok^{\times})^{r} \cap \mathcal{O}_{\CC_{v}}^{r}$.
Then the following are equivalent.
\begin{enumerate}

\item[(i)] The following $v$-adic values are simultaneously vanishing:
\[\Li_{(s_{1},\ldots,s_{r})}(u_{1},\ldots,u_{r})_{v}=\Li_{(s_{2},\ldots,s_{r})}(u_{2},\ldots,u_{r})_{v}=\cdots=\Li_{s_{r}}(u_{r})_{v}=0      .\]

\item[(ii)] The following $v$-adic values are simultaneously vanishing:
\[\Li_{(s_{r},\ldots,s_{1})}^{\star}(u_{r},\ldots,u_{1})_{v}=\Li_{(s_{r},\ldots,s_{2})}^{\star}(u_{r},\cdots,u_{2})_{v}=\ldots=\Li_{s_{r}}^{\star}(u_{r})_{v}=0      .\]

\item[(iii)] $\bv$ is an $\FF_{q}[t]$-torsion point in $G(\ok)$.

\end{enumerate}
\end{theorem}

\begin{proof}
(ii)$\Rightarrow$(i)
This follows from Lemma \ref{lem-CMPL-CMSPL} and Definition \ref{analytic-continuation-CMPLs}.

(i)$\Rightarrow$(ii)
This can be done inductively by using Lemma \ref{lem-CMPL-CMSPL} and Definition \ref{analytic-continuation-CMPLs}:
\begin{eqnarray*}
& & \Li^{\star}_{s_{r}}(u_{r})_{v}  = \Li_{s_{r}}(u_{r})_{v} = 0, \\
& & \Li^{\star}_{(s_{r}, s_{r-1})}(u_{r}, u_{r-1})_{v}
= \Li_{s_{r-1}}(u_{r-1})_{v} \Li^{\star}_{s_{r}}(u_{r})_{v} - \Li_{(s_{r-1}, s_{r})}(u_{r-1}, u_{r})_{v} = 0, \\
& & \Li^{\star}_{(s_{r}, s_{r-1}, s_{r-2})}(u_{r}, u_{r-1}, u_{r-2})_{v} \\
& & = \Li_{s_{r-2}}(u_{r-2})_{v} \Li^{\star}_{(s_{r}, s_{r-1})}(u_{r}, u_{r-1})_{v}
-\Li_{(s_{r-2}, s_{r-1})}(u_{r-2}, u_{r-1})_{v} \Li^{\star}_{(s_{r})}(u_{r})_{v} \\
& & + \Li_{(s_{r-2}, s_{r-1}, s_{r})}(u_{r-2}, u_{r-1}, u_{r})_{v} = 0, \\
& & \vdots \\
& & \Li^{\star}_{(s_{r}, \ldots, s_{1})}(u_{r},\ldots,u_{1})_{v} \\
& & = \sum_{\ell=2}^{r} (-1)^{\ell} \Li_{(s_{1}, \ldots, s_{\ell-1})}(u_{1}, \ldots, u_{\ell-1})_{v} \Li^{\star}_{(s_{r}, \ldots, s_{\ell})}(u_{r}, \ldots, u_{\ell})_{v} \\
& & + (-1)^{r+1} \Li_{(s_{1}, \ldots, s_{r})}(u_{1}, \ldots, u_{r})_{v} = 0.
\end{eqnarray*}

(ii)$\Rightarrow$(iii)
We keep the notation in the proof of Proposition \ref{proposition-log-converge}.
Fix a nonzero polynomial $a \in \FF_{q}[t]$ such that $\bv_{r} := [a] (\bv) \in G(\fm_{v})$.
Since the operator $\partial [v(t)] - v I_{d_{r}}$ on $\Lie \bC^{\otimes d_{r}}$
is nilpotent,
we have $\partial [v(t)^{p^{n}}] = (\partial [v(t)])^{p^{n}} = v^{p^{n}} I_{d_{r}}$ for $p^{n} \geq d_{r}$.
Thus
\[
\partial [v(t)^{p^{n}}](\log_{\bC^{\otimes d_{r}}}(\pi_{r}(\bv_{r})))
\]
lies in the convergence domain of $\exp_{C^{\otimes d_{r}}}$ for $n \gg 0$,
which coincides with
\[
\log_{\bC^{\otimes d_{r}}}([v(t)^{p^{n}}] (\pi_{r}(\bv_{r})))
\]
by Lemma \ref{lem_log_commute}.
Note that by Theorem~\ref{eq_log_value} and Remark ~\ref{remark-analytic-continuation-CMSPLs} the last coordinate of this vector is $v^{p^{n}} a(\theta) \Li^{\star}_{s_{r}}(u_{r})_{v} = 0$.
It follows by \cite[Thm.~3.7]{Yu91} that $\log_{\bC^{\otimes d_{r}}}([v(t)^{p^{n}}] (\pi_{r}(\bv_{r})))$ has to be zero.
Since $\log_{C^{\otimes d_{r}}}$ is the formal inverse of $\exp_{C^{\otimes d_{r}}}$, we have
\[
\pi_{r}([v(t)^{p^{n}}] (\bv_{r})) = [v(t)^{p^{n}}](\pi_{r}(\bv_{r})) = 0.
\]
Thus we conclude that there exists $n_{r}\in \NN$ such that
\[
\bv_{r-1} := [v(t)^{p^{n_{r}}}](\bv_{r}) \in \Ker \pi_{r} =G_{r-1}(\ok).
\]

Repeating this argument, there exist $n_{r-1}, \dots, n_{1}$ such that
\begin{eqnarray*}
\bv_{r-2} &:=& [v(t)^{p^{n_{r-1}}}](\bv_{r-1}) \in \Ker \pi_{r-1} = G_{r-2}(\ok) \\
&\vdots& \\
\bv_{1} &:=& [v(t)^{p^{n_{2}}}](\bv_{2}) \in \Ker \pi_{2} = G_{1}(\ok) \\
\bv_{0} &:=& [v(t)^{p^{n_{1}}}](\bv_{1}) \in \Ker \pi_{1} = G_{0}(\ok).
\end{eqnarray*}
Hence $\bv$ is an $\FF_{q}[t]$-torsion point.

(iii)$\Rightarrow$(ii) Let $a$ be a nonzero polynomial in $\FF_{q}[t]$ chosen as above.
Suppose that there exists a nonzero polynomial $b \in \FF_{q}[t]$
for which $[b]([a] (\bv)) = 0$.
By Lemma \ref{lem_log_commute}, Theorem \ref{eq_log_value} and Remark \ref{remark-analytic-continuation-CMSPLs} we have
\[
0 = \log_{G} ([b]([a](\bv))) = \partial [b](\log_{G} ([a] (\bv)))
= \begin{array}{rcll}
\ldelim( {15}{4pt}[] & * & \rdelim) {15}{4pt}[] & \rdelim\}{4}{10pt}[$d_{1}$] \\
& \vdots & & \\
& * & & \\
& a(\theta) b(\theta)(-1)^{r-1} \Li_{(s_{r}, \dots, s_{1})}^{\star}(u_{r}, \dots, u_{1})_{v} & & \\
& * & & \rdelim\}{4}{10pt}[$d_{2}$] \\
& \vdots & & \\
& * & & \\
&a(\theta) b(\theta) (-1)^{r-2} \Li_{(s_{r}, \dots, s_{2})}^{\star}(u_{r}, \dots, u_{2})_{v} & & \\
& \vdots & & \vdots \\
& * & & \rdelim\}{4}{10pt}[$d_{r}$] \\
& \vdots & & \\
& * & & \\
&a(\theta) b(\theta) \Li_{s_{r}}(u_{r})_{v} & & \\[10pt]
\end{array},\] whence the desired result as $a(\theta) b(\theta)$ is nonzero.

\end{proof}

\begin{corollary}\label{cor-infinite-v}
Let $\fs=(s_{1},\ldots,s_{r})\in \NN^{r}$ and $\bu=(u_{1},\ldots,u_{r})\in (\ok^{\times})^{r} \cap \mathcal{O}_{\CC_{v}}^{r}$.
Assume that $|u_{i}|_{\infty} < |\theta|_{\infty}^{\frac{s_{i} q}{q-1}}$ for each $i$.
Then $\Li_{\fs}(\bu) / \tilde{\pi}^{s_{1} + \cdots + s_{r}} \in \ok$ if and only if
\[
\Li_{(s_{1}, \dots, s_{r})}(u_{1}, \dots, u_{r})_{v}
= \Li_{(s_{2}, \dots, s_{r})}(u_{2}, \dots, u_{r})_{v}
= \cdots
= \Li_{s_{r}}(u_{r})_{v}=0.
\]
\end{corollary}

\begin{proof}
We have that $\Li_{\fs}(\bu) / \tilde{\pi}^{s_{1} + \cdots + s_{r}} \in k$ if and only if $\Li_{\fs}(\bu) / \tilde{\pi}^{s_{1} + \cdots + s_{r}} \in \ok$ by \cite[Thm.~5.4.3]{C14}.
By \cite[Thm.~4.3.2 and Rmk.~4.3.3]{CPY14} and Remark~\ref{Rem:G and v}, $\Li_{\fs}(\bu) / \tilde{\pi}^{s_{1} + \cdots + s_{r}} \in k$
if and only if $\bv_{\fs,\bu}$ is an $\FF_{q}[t]$-torsion point in $G(\ok)$.
Hence by Theorem \ref{T:VanishingCriterion} we have the desired result.
\end{proof}


\bibliographystyle{alpha}

\end{document}